\documentclass[11pt, reqno]{amsart}

\usepackage{amsmath,amssymb,amsfonts,amsthm}

\newtheorem{theorem}{Theorem}
\newtheorem*{theorem*}{Theorem}

\newtheorem{remark}{Remark}
\newtheorem{lemma}{Lemma}
\newtheorem*{lemma*}{Lemma}
\newtheorem{definition}{Definition}

\newenvironment{preuve}{\noindent {\it Proof}}{\hfill$\square$}

\usepackage{color}                    
\usepackage[colorlinks,citecolor=blue]{hyperref}
\definecolor{uququq}{rgb}{0.25,0.25,0.25}

\usepackage{graphics}
\usepackage{pgf,tikz}
\usetikzlibrary{arrows} 
\usepackage[utf8]{inputenc}

\newcommand{\ga}{\gamma }

\newcommand\ep{\varepsilon}
\newcommand{\e}{\varepsilon}
\newcommand\ph{\varphi}

\newcommand{\pp}{{\mathbb{P}}}
\newcommand{\PP}{{\mathbb{P}}}
\newcommand{\E}{{\mathbb{E}}}
\newcommand{\EE}{{\mathbb{E}}}
\newcommand{\indiq}{{\bf 1}}

\newcommand{\rr}{{\mathbb{R}}} 
\newcommand{\RR}{{\mathbb{R}}} 
\newcommand{\R}{{\mathbb{R}}}
 
\newcommand{\NN}{{\mathbb{N}}}
\newcommand{\N}{{\mathbb{N}}}

\newcommand{\ZZ}{\mathbb{Z}}
\newcommand{\Z}{\mathbb{Z}}

\newcommand{\cA}{{\mathcal A}}

\newcommand{\cQ}{{\mathcal Q}}
\newcommand{\cR}{{\mathcal R}}
\newcommand{\cS}{{\mathcal S}}

\newcommand{\cZ}{{\mathcal Z}}

\newcommand{\wt}{\widetilde}

\def \( {\left(}
\def \) {\right)}
\def \[ {\left\lbrack}
\def \] {\right\rbrack}

\def \llb {\left\lbrace}
\def \rrb {\right\rbrace}
\def \< {\left\langle}
\def \> {\right\rangle}
\def\sk{\smallskip}
\def\mk{\medskip}

\newcommand{\Dimh}{\mathrm{Dim}_\mathrm{H}}
\newcommand{\Dimum}{\mathrm{Dim}_\mathrm{UM}}
\newcommand{\Dimlm}{\mathrm{Dim}_\mathrm{LM}}

\newcommand{\pix}{\mathrm{pix}}

\providecommand{\proo}[1]{\{#1_t: \, t \geq 0\}}

\begin{document}
\title[Brownian sojourn inside moving boundaries]{On sojourn of Brownian motion inside moving boundaries}
\author[S. Seuret \and X. Yang]{St\'ephane Seuret \and Xiaochuan Yang}
\address{St\'ephane Seuret, Universit\'e Paris-Est\\ LAMA(UMR8050), UPEMLV, UPEC, CNRS \\ F-94010 Cr\'eteil \\France. }
\email{seuret@u-pec.fr}

\address{Xiaochuan Yang, Universit\'e Paris-Est\\ LAMA(UMR8050), UPEMLV, UPEC, CNRS \\ F-94010 Cr\'eteil, France.  {\rm Present address: Dept. Statistics \& Probability, Michigan State University, 48824 East Lansing, MI, USA}}

\email[]{xiaochuan.j.yang@gmail.com}
\date{\today}
\keywords{60G51, 28A80, 60G17: Macroscopic Hausdorff dimension, mass dimension, Brownian motion, sojourn sets.}

\begin{abstract} We investigate the large scale structure of certain sojourn sets of one dimensional Brownian motion within two-sided moving boundaries.  The macroscopic Hausdorff dimension and upper mass dimension of these sets are computed.  We also give a uniform macroscopic dimension result for the Brownian level sets.   
\end{abstract}

\maketitle

\section{Introduction}

Let $B=\proo B$ be a real-valued standard Brownian motion starting from $0$. This article is concerned with the sojourn properties of $B$ within some moving boundaries. More precisely,  for an appropriate function $\ph : \R^+\to \R^+$,   consider the sets 
\begin{align}
E(\ph) = \llb t\ge 0 :  |B_t|\le \ph(t) \rrb  \label{inside}.
\end{align} 
that we call the set of Brownian sojourn within the two-sided boundary $\ph(\cdot)$.

Besides its obvious application in physics and finance, the understanding of these sets {at} different scales entails considerable information on the path properties of the Brownian motion.  Two types of study are of particular interest.  

\mk

\noindent $\bullet$ Geometric properties  of $E(\ph)$ near $t_0=0$.  This corresponds to the regular behavior of a Brownian path near zero. Concretely, local asymptotic laws of the Brownian motion such as Khintchine's law of iterated logarithm can be described in terms of geometric properties of these sets around $0$ with specific choices for $\ph$.   A natural question is under which condition on $\ph$ these sets admit an  upper density with respect to the Lebesgue measure (denoted by $|\cdot|$ throughout the paper),  namely,  $$\limsup_{s\to 0} \frac{|E(\ph)\cap[0,s] |}{s} = c_\ph>0.$$   Uchiyama \cite{uchiyama1982} treated the case $\ph(t) = h(t)\sqrt{t}$ where $h$ is taken from  a whole class of correction functions (of logarithmic order), and he provided an integral test for $c_\ph$ to be positive.

\mk

\noindent $\bullet$ Geometric properties {of $E(\ph)$}  at infinity.  This is related to the long time behavior of the Brownian motion.  As {$B$} behaves like a square root function at infinity,  the set $E(\ph)$, when $\ph$ grows slower than the square root function, concerns the lower than normal growth of the Brownian motion.  Uchiyama \cite{uchiyama1982} considered the upper density as above for $\ph(t) = \sqrt{t}/h(t)$ with $h$ belonging to a large class of {correction functions with logarithmic order}. 

\mk
Let us mention briefly other related works. Consider
$$
\wt E(\ph) = \{ t\ge 0:  |B_t|\ge \ph(t) \}.
$$
where $\ph$ grows like a square root function with a logarithmic order correction. The geometry of these sets around zero describes the local irregular behavior of $B$, whereas their geometry around infinity describes the distribution of high peaks of $B$ - we refer the interested reader to Strassen \cite{strassen1964} and Uchiyama \cite{uchiyama1982} for upper density results (around zero and infinity). We also mention the recent work of Khoshnevisan, Kim and Xiao \cite{khoshnevisan2015multifractal} and Khoshnevisan and Xiao \cite{khoshnevisanXiao2016} who considered, among other things, the high peaks of symmetric stable L\'evy processes. {We refer to \cite{CiesielskiTaylor62, Ray63, Berman88, Bermann91, BeghinNikitinOrsingher03} for other results on sojourn properties of various stochastic processes.}

\mk
Motivated by these studies,  we focus on the asymptotic behavior around infinity of  the sojourn sets of Brownian motion within moving boundaries with \emph{much lower than normal growth}. For this, we introduce the sets 
\begin{align*}
\forall\, \ga\in[0,1/2], \quad E_\gamma: =  E(\ph_\ga)   \mbox{ with }   \ph_\ga(t) = t^{\ga}, \quad t\ge 0, 
\end{align*}
and our goal is to estimate the size of $ E_\gamma$ for all values of $\ga$. 
This size is expressed in terms of  three different quantities, including  the notion of  large scale dimensions developed by Barlow and Taylor \cite{barlow1989, barlow1992}  in the late 80's,  recently ``refreshed" in the work of Xiao and Zheng \cite{xiao2013}, Georgiou {\it et al.} \cite{georgiou},  Khoshnevisan {\it et al.} \cite{khoshnevisan2015multifractal, khoshnevisanXiao2016}.

\mk

The initial motivation of Barlow and Taylor was to define a notion of fractals in discrete spaces such as $\Z^d$ in order to characterize  the size of models in statistical physics, e.g. the infinite connected component of a percolation process, the range of a transient random walk etc.  To this end, they  introduced and investigated several notions of  dimension  describing different types of asymptotic behavior of a set around infinity.  Each dimension corresponds to an analog in large scales of a classical fractal dimension.   Among these dimensions,  we are going to use the  ``macroscopic Hausdorff dimension" and  the ``mass dimensions", which are  analogs of the classical Hausdorff dimension   and  Minkowski dimensions in large scales, respectively. 

Let us state our first main theorem.   The macroscopic Hausdorff dimension of a set $E\subset \R^d$ is denoted by $\Dimh E$,  the upper mass dimension by $\Dimum E$, and the logarithmic density by $\mathrm{Den}_{\log}$. The  definitions are recalled in Section~2.

\begin{theorem}\label{main} Almost surely,  for all $\ga\in[0,\frac 1 2]$,
\begin{align}  
\Dimh E_\gamma &= \begin{cases}  \frac{1}{2} & \mbox{ if } \ga\in[0,\frac 1 2), \\ 1 & \mbox{ if } \ga=\frac 1 2 . \end{cases} \label{dimh} \\ 
\Dimum E_\gamma & = \mathrm{Den}_{\log} E_\ga  =\frac{1}{2} + \ga. \label{dimum} 
 \end{align}
\end{theorem}

\begin{remark} \ 

\begin{itemize}
\item Since the standard scale of Brownian motion at time $t$ is $\sqrt{t}$ and Brownian paths fluctuate not so violently in the scale $\sqrt{t}$,  it is expected that the fractal dimension of $E_\ga$ does not depend on $\ga$ if $\ga<\frac 1 2$.   This intuition is confirmed by \eqref{dimh} and can be interpreted as such,  $B$ spends most of the time at its ``boundary" from the macroscopic Hausdorff dimension standpoint.
\item 
Although they have the same macroscopic Hausdorff dimension, the sets $E_\gamma$, $0\le \ga<1/2$, differ by their upper mass dimensions. Indeed, \eqref{dimum} indicates a multifractal behavior for $\Dimum E_\gamma$,  and  gives a natural example of sets for which the macroscopic Hausdorff dimension and the upper mass dimension differ.
\item   The Brownian motion $B$ returns inside the boundary $\ph_\ga$ infinitely many times, for every $\gamma\in [0,1/2]$, so the sets $E_\gamma$ are unbounded.  Our result   quantifies   recurrence and fluctuation properties of $B$ at large scales. 
\end{itemize}
\end{remark}

Our contribution to the formula \eqref{dimh} is the equality when $\ga\in[0,1/2)$.  The case $\ga=\frac 1 2$ is deduced from Theorem 2 in \cite{uchiyama1982} where Uchiyama obtained that 
$$\mbox{ a.s. } \ \ \ \ \limsup_{r\to+\infty}  \frac{|E(\frac 1 2)\cap[0,r]|}{r} > 0.$$  This inequality  entails that a.s. $\Dimh E(\frac 1 2) =1$ thanks to the following fact proved in \cite{khoshnevisan2015multifractal} : for any $E\subset \rr$,
$$  \limsup_{r\to+\infty}  \frac{|E\cap[0,r]|}{r} > 0\ \ \Rightarrow  \ \Dimh E = 1.$$

\mk
 
Our second result is a uniform dimension result for the level sets
$$\cZ_x= \{t\geq 0: B_t=x\}.$$ It turns out that, though much thinner than any $E_\ga$,  the level sets all have the same macroscopic Hausdorff  dimension as that of $E_\ga$, for $0<\ga<\frac 1 2$. 

\begin{theorem}\label{main2} {Almost surely, for all $x\in\RR$,  $ \Dimh(\cZ_x)=\frac 1 2$. } \end{theorem}

The local structure of $\cZ_x$ is well understood since the works by Taylor and Wendel \cite{taylor1966} and Perkins \cite{perkins1981}, who proved Hausdorff measure results for $\cZ_x$ using local times for fixed $x$ and for all $x$, respectively.  In particular,  the classical Hausdorff dimension of $\cZ_0$ is $\frac 1 2$, a result already known in \cite{blumenthal1960stable}. Our result gives the large scale structure of $\cZ_x$ and  might be compared with an interesting result by Khoshnevisan \cite{khoshnevisan1994} which states that the zero set of a symmetric random walk {$\{\xi_n: \,n\in\NN\}$} in $\Z^1$   with finite variance  
 $$\{ n \in \N :  \xi_n = 0\}$$
 has macroscopic Hausdorff dimension $\frac 1 2$.

\medskip
 
This paper is organized as follows.   In Section \ref{sec2}, we recall the definition of large scale dimensions and  establish some preliminary estimates.  
Theorem \ref{main} is proved in Section \ref{sec: proof  DimH}. Theorem \ref{main2}  is proved in Section \ref{sec_main2}.

\section{Preliminaries}\label{sec2}

Throughout the paper,  $c, C$ are generic positive finite constants whose value may change from line to line.  For two  families of  positive real numbers  $(a(x))$ and $(b(x))$, the equation $a(x)\asymp b(x)$ means that the ratio $a(x)/b(x)$ is uniformly bounded from below and above by some positive finite constant independent of $x$. The set of non-negative real numbers is denoted by $\RR^+$ and the set of positive integers by $\NN^*$.  Also,  $\PP^x$ denotes the law of Brownian motion starting from $x\in\RR$, and $\EE^x$ {denotes} the expectation with respect to $\PP^x$. For simplicity, we write $\PP=\PP^0$  {and $\EE=\EE^0$}.  

\subsection{Macroscopic dimensions}
We adopt the notations in \cite{khoshnevisan2015multifractal,khoshnevisanXiao2016} that we recall now.  Set $Q(x,r) = [x,x+r)$ for all $x\in\RR$ and $r>0$.  

{Define the annuli $\forall\, n \ge 1, \cS_n = [2^{n-1}, 2^{n})$ and  $\cS_0 = [0,1)$. For any $\rho\geq 0$,  any set $E\subset \R^+$,  $n\in\N^*$, we introduce the quantity
\begin{align}
\label{defnurho}
\nu_\rho^n (E)  &=  \inf\llb \sum_{i=1}^m \( \frac{|Q_i|}{2^n}\) ^\rho :   Q_i\subset \cS_n, \ \  E\cap\cS_n \subset  \bigcup_{i=1}^m Q_i  ,\right.\\
\nonumber&   \ \ \ \ \ \ \ \ \  \ \ \ \   \mbox{ and } Q_i \mbox{ is a non-trivial interval with  integer boundaries    }  \Big\} .
\end{align}
Other gauge functions   could be used instead of $x\mapsto x^\rho$.   
 
The key constraint is that the sets  $Q_i$  used in \eqref{defnurho} are  non-trivial intervals with integer bounds. In particular, their length is always greater than 1.
 
\begin{definition} Let $E\subset \R^+$.  The macroscopic Hausdorff dimension of $E$ is defined as
\begin{equation} \label{eq_defdimh}
\Dimh E =  \inf\llb \rho\ge 0 :  \sum_{n\ge 0} \nu^n_\rho(E)<+\infty  \rrb . 
\end{equation}
\end{definition}

We move to the definition of mass dimensions and logarithmic density. For any $E\subset \RR^+$, define the pixelization of $E$ by the set of integers that are of distance at most $1$ from $E$, namely, 
\begin{align*}
\mbox{pix}(E) = \{x\in\ZZ:  \mbox{dist}(x, E)\le 1 \}.
\end{align*}

\begin{definition}
The upper and lower mass dimension of $E$ are defined as
\begin{align*}
\Dimum E = \limsup_{n\to+\infty} \frac{\log_2(\sharp\, \pix(E\cap[0,2^n])}{n}, \\
\Dimlm E = \liminf_{n\to+\infty} \frac{\log_2(\sharp\, \pix(E\cap[0,2^n])}{n}. 
\end{align*}
The logarithmic density of $E$ is defined as
\begin{align*}
\mathrm{Den_{\log}} E =  \limsup_{n\to+\infty} \frac{\log_2 |E\cap [0,2^n]|}{n}.
\end{align*}
\end{definition}
Recall that $|A|$ stands for the Lebesgue measure of $A$.
The macroscopic Hausdorff dimension of a set does not depend on  any of its  bounded subsets, since the series in \eqref{eq_defdimh} converges if and only if its tail series converges.  Further, the covering intervals are chosen to have {length} larger than $1$, which explains why the macroscopic Hausdorff dimension does not rely on the local structure of the underlying set.  The following fact  is known \cite{barlow1989,barlow1992}:  for any set $E\subset\R$,   $$\Dimh E \le \Dimlm E \le \Dimum E.$$ 
Since the cardinal of the pixelization of a set $E$ is always larger than the Lebesgue measure of $E$, we also have
$$\mathrm{Den}_{\log} E \le \Dimum E.$$

To bound  $\Dimh E$ from above,  one usually exhibits an economic covering of $E$.  To get a lower bound,    the following lemma, which is an analog of the mass distribution principle, is   useful. 
\begin{lemma}\label{mdp}  Let $E\subset\cS_n$, $n\in\NN$.  Let $\mu$ be a finite Borel measure on $\R$ with support included in $E$.  Suppose that there exists finite positive  constants $C$ and $\rho$, such that for any interval $Q\subset \cS_n$ with $|Q|\ge 1$,  one has 
$$\mu(Q) \le C |Q|^\rho.$$
Then $$\nu_\rho^n(E) \ge  C^{-1} 2^{-n\rho} \mu(\cS_n). $$
\end{lemma}

In fact, there is more flexibility in the choice of covering intervals {when we are only concerned with the value of  $\Dimh E$}.  Let us introduce, for every integer $n\geq 1$ and any set $E\subset \R^+$,  the quantity 
\begin{align}
\nonumber \widetilde\nu_\rho^n (E) &=  \inf\llb \sum_{i=1}^m \( \frac{|Q_{i}|}{2^n}\) ^\rho  :     Q_{i}\subset \cS_n  , \ \ E\cap\cS_n \subset  \bigcup_{i=1}^m Q_{i} , \ \   \frac{|Q_{i}|}{n^2} \in \NN^*  , \right . \\ 
&   \ \ \ \ \ \ \ \ \  \ \ \ \   \mbox{ and } Q_i \mbox{ is an interval with  integer boundaries    }  \Big\} .
\label{defnurho2}
\end{align}
The difference with $\nu_\rho^n$ is that  coverings by sets of size that are multiple of $n^2$ are used, and this does not change the value of the macroscopic Hausdorff dimension, as stated by the following lemma.
\begin{lemma}
\label{lem13}
For every set $E\subset \R^+$,
\begin{equation} \label{eq_defdimh2}
\Dimh E =  \inf\llb \rho\ge 0 :  \sum_{n\ge 0} \widetilde \nu^n_\rho(E)<+\infty  \rrb . 
\end{equation}
\end{lemma}
\begin{proof}
Let us denote by $\widetilde d$ the value in the right hand-side of \eqref{eq_defdimh2}.

Obviously, for every $\rho \geq 0$, $\widetilde\nu_\rho^n (E) \geq \nu_\rho^n (E)$. Hence,  recalling \eqref{eq_defdimh},  $\Dimh E\leq \widetilde d$. {If $\Dimh(E)=1$, then \eqref{eq_defdimh2} holds trivially. Assume thus $\Dimh(E)<1$.}

Let $1>\rho'>\rho >\Dimh (E)$, and fix $n\geq 1$.  Choose  $\mathcal Q_n:=\{Q_i\}_{i=1,..., m}$ as a finite family of   intervals such that $E\cap\cS_n \subset  \bigcup_{i=1}^m Q_i$,  with $ |Q_i|\ge 1$ and $ \sum_{i=1}^m \( \frac{|Q_i|}{2^n}\) ^\rho  \leq 2\nu_\rho^n (E) $.

Define the finite family of   intervals  $\widetilde{\mathcal Q}_n:=\{\widetilde Q_i\}_{i=1,..., m}$ as follows: $\widetilde Q_i$ is an interval containing $Q_i$, included in $\mathcal{S}_n$, whose length is the smallest possible multiple of $n^2$.

Observe that if $|Q_i|\geq n^2$, $|\widetilde Q_i|\leq 2|Q_i|$, while if $1\leq |Q_i|< n^2$, $|\widetilde Q_i|\leq n^2|Q_i|$.

By construction, $E\cap \cS_n  \subset \bigcup_{i=1}^m \widetilde Q_i$. In addition,  since $\rho<\rho'<1$,
\begin{align*}
\sum_{i=1}^m \( \frac{|\widetilde Q_i|}{2^n}\) ^{\rho '} &= \sum_{i=1:|Q_i| <n^2}^m \( \frac{|\widetilde Q_i|}{2^n}\) ^{\rho '} +\sum_{i=1:|Q_i| \geq n^2}^m \( \frac{|\widetilde Q_i|}{2^n}\) ^{\rho '} \\
& \leq   \sum_{i=1:|Q_i| <n^2}^m  n^2\( \frac{|Q_i|}{2^n}\) ^{\rho '} + 2 \sum_{i=1:|Q_i| \geq 
n^2}^m \( \frac{|Q_i|}{2^n}\) ^{\rho '} \\
& \leq    \frac{n^{2+\rho'-\rho} }{2^{n(\rho'-\rho)}}  \sum_{i=1:|Q_i| <n^2}^m    \( \frac{|Q_i|}{2^n}\) ^{\rho } +2  \sum_{i=1:|Q_i| \geq 
n^2}^m \( \frac{|Q_i|}{2^n}\) ^{\rho} .
\end{align*}

When $n$ becomes large, $\frac{n^{2+\rho'-\rho} }{2^{n(\rho'-\rho)}}  \leq 1$, hence  $$\sum_{i=1}^m \( \frac{|\widetilde Q_i|}{2^n}\) ^{\rho '} \leq   {2}\sum_{i=1}^m    \( \frac{|Q_i|}{2^n}\) ^{\rho } .$$
One deduces that  $\widetilde\nu_{\rho'}^n(E) \leq 2 \nu^n_\rho (E)$. Since $\sum_{n\geq 1}\nu^n_\rho (E)<+\infty$, the series $\sum_{n\geq 1}  \widetilde\nu_{\rho'}^n(E)$ also converges, and $\widetilde d \leq \rho'$. Letting $\rho'$ tend to $\Dimh (E)$ yields the result.
\end{proof}

\begin{remark}
\label{rem_3}
{The same argument shows that for any $C\geq 1$, using coverings of a set $E\cap \cS_n$ by sets of size larger than $C$ or than $Cn^2$   instead $1$ or $n^2$} in the definitions  \eqref{defnurho} or \eqref{defnurho2},  does not change the value of the macroscopic Hausdorff dimension of $E$.
\end{remark}

\subsection{Hitting probability estimates of Brownian motion inside the moving boundaries }   The following estimate is useful when looking for an appropriate covering of $E_\gamma$ with respect to  different large scale dimensions.

\begin{lemma}\label{hittingbm}
Consider an interval $Q(a,r)$ inside $\cS_n$, i.e. $a\ge 2^{n-1}$ and $a+r\le 2^n$.  For each $0\le \ga< 1/2$,  define the event 
\begin{align}\label{eventbm} 
{\cA({n,a,r, \ga})} = \llb \exists\, t\in Q(a,r) :  |B_t|\le t^{\ga} \rrb .
\end{align}
One has
\begin{align*}
\pp \( \cA({n,a,r,\ga}) \) \le  \frac{2}{\sqrt{\pi}} 2^{n(\ga-1/2)}  +  \frac{4}{\sqrt{2\pi}} \( \frac{r}{a} \) ^{1/2}.
\end{align*}
\end{lemma}

{Following a standard vocabulary, the event $\cA({n,a,r, \ga}) $ describes the hitting probability of $B$ inside the moving boundary $\ph_\ga$. Lemma \ref{hittingbm} states that, when $\frac{r}{a}$ is small, this  hitting probability depends only on $\gamma$ and $n$, while it behaves like  $\left(\frac{r}{a}\right)^{1/2}$ when $\frac{r}{a}$ becomes large.}  Basic properties of Brownian motion used in the proof of Lemma \ref{hittingbm} can be found in \cite{revuzyor1999} or \cite{morters2010}.

\mk
\begin{preuve}{\it  \ of Lemma \ref{hittingbm} : }  One has
\begin{align*}
\pp(\cA({n,a,r,\ga}))  &\le \pp \( \inf_{t\in Q(a,r)}|B_t| \le 2^{n\ga} \) \\
&= \pp\( |B_a| \le 2^{n\ga} \) + \pp\( |B_a| > 2^{n\ga},  \inf_{t\in Q(a,r)} |B_t| \le 2^{n\ga} \)  \\
&:= P_1+P_2. 
\end{align*}
By the self-similarity of $B$ and recalling that $a\in\cS_n$, one obtains
\begin{align*} P_1 = \pp(|B_1| \le a^{-1/2} 2^{n\ga}) \le \sqrt{\frac{2}{\pi}} a^{-1/2} 2^{n\ga} \le \frac{2}{\sqrt{\pi}} 2^{n(\ga-1/2)}. 
\end{align*} 
Using the symmetry of $B$,  one gets
\begin{align*}
\pp\( B_a > 2^{n\ga},  \inf_{t\in Q(a,r)} B_t  \le 2^{n\ga} \) =  \pp\(  B_a < - 2^{n\ga},  \sup_{t\in Q(a,r)} B_t  \ge -2^{n\ga} \) . 
\end{align*}
Thus, 
\begin{align*}
P_2 &\le 2 \pp\( B_a > 2^{n\ga},  \inf_{t\in Q(a,r)} B_t  \le 2^{n\ga} \) \\
&= 2 \pp \( B_a>2^{n\ga}, \inf_{t\in Q(a,r)} (B_t-B_a) \le 2^{2n\ga} -B_a \) .
\end{align*}
Set {$\tilde B _{h} = B_{a+h}- B_a$} which is a Brownian motion independent of $B_a$. Using successively the self-similarity, the symmetry and the Markov property yields
\begin{align*}
P_2 &\le 2 \pp\( B_a > 2^{n\ga}, \inf_{0\le h\le r} \tilde B_h \le 2^{n\ga} -B_a \) \\
&= 2 \pp\( B_a > 2^{n\ga}, \inf_{0\le h\le 1} \tilde B_h \le r^{-1/2} (2^{n\ga} -B_a) \) \\
&=2  \pp\( B_a > 2^{n\ga}, \sup_{0\le h\le 1} \tilde B_h \ge r^{-1/2} (B_a - 2^{n\ga} )  \) \\
&= 2\int_{2^{n\ga}}^{+\infty}  \pp\( \sup_{0\le h \le 1} \tilde B_h \ge r^{-1/2} (x-2^{n\ga}) \)  e^{-\frac{x^2}{2a}} \frac{dx}{\sqrt{2\pi a}} :=I_1+I_2,
\end{align*}
where $I_1$ is the integral over the interval $[2^{n\ga}, 2^{n\ga}+r^{1/2}]$, and $I_2$ over $[2^{n\ga}+r^{1/2},\infty)$.  On one hand, bounding from above the probability in the integrand by $1$, one obtains  
\begin{align*}
I_1\leq 2\int_{2^{n\ga}}^{2^{n\ga}+r^{1/2}}  e^{-\frac{x^2}{2a}}  \frac{dx}{\sqrt{2\pi a}}  \le \frac{2}{\sqrt{2\pi}} \( \frac{r}{a} \) ^{1/2} .
\end{align*}
On the other hand, one knows by applying the reflection principle to $\tilde B$ that $\sup_{0\le h \le 1} \tilde B_h$ has the same distribution as $|\tilde B_1|$. Hence, using the tail estimate for a standard normal random variable \cite[page 192]{marcusRosen2006}, one has
\begin{align*}
{I_2 \ }& \leq  4\int_{2^{n\ga}+r^{1/2}}^{+\infty} \pp\(  \tilde B_1 \ge r^{-1/2}(x-2^{n\ga}) \) e^{-\frac{x^2}{2a}} \frac{dx}{\sqrt{2\pi a}} \\
&\le 4  \int_{2^{n\ga}+r^{1/2}}^{+\infty}  \frac{1}{r^{-1/2}(x-2^{n\ga})\sqrt{2\pi}}  e^{-\frac{(x-2^{n\ga})^2}{2r}} e^{-\frac{x^2}{2a}} \frac{dx}{\sqrt{2\pi a}} \\
&\le4  \frac{1}{\sqrt{2\pi}} \int_{2^{n\ga}}^{+\infty} e^{-\frac{(x-2^{n\ga})^2}{2r}} \frac{dx}{\sqrt{2\pi a}} =  \frac{4}{\sqrt{2\pi}} \( \frac{r}{a} \) ^{1/2}  \int_{2^{n\ga}}^{+\infty}  e^{-\frac{(x-2^{n\ga})^2}{2r}} \frac{dx}{\sqrt{2\pi r}} \\
&= \frac{2}{\sqrt{2\pi}} \( \frac{r}{a} \) ^{1/2} .
\end{align*}
Therefore,  one has established that \begin{align*}
P_2 \le \frac{4}{\sqrt{2\pi}}\( \frac{r}{a}\) ^{1/2}.
\end{align*}
Combining the estimates above ends the proof. 
\end{preuve}

 \subsection{Brownian local times}\label{sec31}

{Let us recall very briefly the notion of local times.}
Let $f$ be a  non-negative continuous even function with integral one. For each $\e>0$, set $f_{\e}(x) = f(x/\e)/\e$. Define the local times at zero of $B$ by
\begin{align*}
L_t = \lim_{\e\to 0}\frac{1}{2\e} \int_0^t f_{\e}(B_s) {\rm d}s.
\end{align*}
It is known \cite[Lemma 2.4.1]{marcusRosen2006} that the convergence occurs uniformly on $t\in[0,T]$,  $\PP^y$ almost surely, for any $y\in\RR$ and $T>0$. From the definition we see that $t\mapsto L_t$ extends to a measure, denoted by $L(\mathrm{d} t)$, which is supported on the zero set of Brownian motion  \cite[Remark 3.6.2]{marcusRosen2006}.

 Next lemma gives the asymptotic behavior at infinity of the local time increments, which is the continuous analog of Corollary 2.2 in \cite{khoshnevisan1994}.   

\begin{lemma}\label{lemma: increments of LT} One has $\PP$-a.s. 
\begin{align*}
\limsup_{n\to\infty} n^{-1} \sup_{t\le 2^{n} \atop t\in\NN }\sup_{2\le h\le 2^{n-2} \atop h\in\NN} {L_{t+h}-L_t \over \sqrt{h/\log_2 h}} \le  4.
\end{align*}
\end{lemma}
\begin{proof}
Set $S_t = \sup_{0\le s\le t} B_s$. {A famous theorem by L\'evy} \cite[page 240]{revuzyor1999} states that the processes $\proo S$ and $\proo L$ have the same law. Also, by  the reflection principle, $S_t=|B_t|$ in distribution for any fixed $t$. Therefore, the classical Gaussian tail estimate \cite[page 192]{marcusRosen2006}  yields that for all $t,x>0$,  
\begin{equation}\label{eq: tail of LT}
\PP(L_t\ge x) = \PP(|B_1|\ge x/\sqrt{t}) \le  e^{-x^2/(2t)}.
\end{equation}
For each $t\ge 0$, define $T_t = \inf\{s\ge t: B_s=0\}$ the first hitting time at zero of Brownian motion after time $t$. Since $L$ only increases when the Brownian motion hits  zero, one has for almost every sample path $\omega$ that, $L_{t+h}-L_t \le L_{T_t+h}-L_{T_t}= L_h\circ \theta_{T_t}$  for any $t,h\ge 0$, see \cite[page 402]{revuzyor1999} for the equality. Here $\theta$ is the usual shift operator on the space of continuous functions.
 
 This, combined with \eqref{eq: tail of LT} and the strong Markov property at the stopping time $T_t$,  implies that for every $x>0$, 
\begin{align*}
 \pp\left( \sup_{t\le 2^{n} \atop t\in\NN }\sup_{2\le h\le 2^{n-2} \atop h\in\NN} {L_{t+h}-L_t \over \sqrt{h/\log_2 h}} \ge x \right) 
&\le 2^{2n-2} \sup_{t\le 2^{n} \atop 2\le h\le 2^{n-2}} \PP\left({L_{t+h}-L_t \over \sqrt{h/\log_2 h}} \ge x \right) \\
&\le 2^{2n-2} \sup_{2\le h\le 2^{n-2}} \PP\left({L_h  \ge x \sqrt{h/\log_2 h}} \right) \\
&\le 2^{2n} e^{-x^2/(2n)}
\end{align*}
{Taking $x= 4n$ in the above inequality, and summing over $n\geq 1$, the Borel-Cantelli lemma yields the conclusion.}
\end{proof}


\section{Proof of Theorem \ref{main} }\label{sec: proof  DimH}

\subsection{Macroscopic Hausdorff dimension of $E_\ga$}
 We prove the dimension formula \eqref{dimh}.   As   said in the introduction, by the upper density result of Uchiyama \cite{uchiyama1982},  it is enough to compute $\Dimh E_\gamma$ for all $\ga\in [0, \frac 1 2)$.  Let $0\le \ga< \frac 1 2$ be fixed throughout this section. Due to the monotonicity in $\ga$ of the sets $E_\gamma$,   and the fact that the zero set of Brownian motion is included in $E_0$,  we divide the proof of \eqref{dimh} into two parts : 
 $$\Dimh E_\gamma \le \frac{1}{2} \ \ \mbox{ and }  \ \Dimh \cZ_0 \ge \frac{1}{2}. $$  

\subsubsection{Proof of $\Dimh E_\gamma \le \frac{1}{2}$}
A first moment argument is used for  this upper bound.  Let  $\rho>\frac 1 2$ and set 
$$x_{n,i} = 2^{n-1}+ i 2^{2n\ga} \mbox{ for } i\in\{0,\ldots, \lfloor 2^{n-1}/2^{2n\ga}\rfloor \}.$$  
Consider the intervals $Q(x_{n,i}, 2^{2n\ga})$ which form a partition of $\cS_n$.  Note that  $E_\gamma\cap Q(x_{n,i}, 2^{2n\ga}) \neq\emptyset$ if and only
if  the event $\cA({n, x_{n,i}, 2^{2n\ga}, \ga})$ occurs
(see its definition in Lemma \ref{hittingbm}). Thus,
\begin{align*}
\nu_\rho^n(E_\gamma) \le {\sum_{i=0}^{\lfloor 2^{(1-2\ga)n-1} \rfloor }} \( \frac{2^{2n\ga}}{2^n}  \) ^\rho \indiq_{\cA({n, x_{n,i}, 2^{2n\ga}, \ga})} .
\end{align*}
By choosing the length $2^{2n\ga}$,  one observes that the two terms in Lemma \ref{hittingbm} are of the same order. Taking expectation in the above inequality, one obtains by Lemma \ref{hittingbm} that there exists a positive finite constant $C$ such that for all $n\in\N^*$ 
\begin{align*}
\E[\nu^n_\rho(E_\gamma)] &\le  2^{(1-2\ga)(1-\rho)n}  \cdot C 2^{n(\ga-1/2)}  \\
&=  C 2^{(1/2 - \ga)(1-2\rho)n}.
\end{align*}
Thus, {the} Fubini Theorem entails $\E[\sum_{n=1}^\infty \nu_\rho^n(E_\gamma)] < + \infty. $  This proves that $\Dimh E_\gamma \le \rho$ almost surely.  Letting $\rho\downarrow 1/2$ yields the upper bound. 

\bigskip

\subsubsection{Proof of  $\Dimh\cZ_0\ge 1/2$} The proof   follows the idea of Khoshnevisan \cite[page 581]{khoshnevisan1994}.

\sk
Let $g(\e)= \sqrt{\e\log_2(1/\e)}$ for $0<\e\le 1/2$. By Lemma  \ref{lemma: increments of LT}, there exists a random  integer $ n_0(\omega)$ such that for all $n\ge n_0(\omega)$, 
\begin{equation}
\label{ineg2}
  \sup_{t\le 2^{n} \atop t\in\NN }\sup_{2\le h\le 2^{n-2} \atop h\in\NN} {L_{t+h}-L_t \over \sqrt{h/\log_2 h}} \leq 5n .
  \end{equation}
{Consider  any finite sequence of intervals $\{Q_{n,i}=[a_{n,i},b_{n,i}]\}_{i=1,...,m}$ }  with integer endpoints and with {sidelength} $2\le |Q_{n,i}|\le 2^{n-2}$, that form a covering of $\cZ_0\cap\cS_n$.

Using \eqref{ineg2}, one obtains
\begin{align*}
 g\left( \frac{|Q_{n,i}|}{2^n}\right) &=  \sqrt{ \frac{|Q_{n,i}|}{2^n}\log_2 \frac{2^n}{|Q_{n,i}|}}\\
 &= 2^{-n/2}\sqrt{\frac{|Q_{n,i}|}{\log_2 |Q_{n,i}|} \left( n-\log_2 |Q_{n,i}| \right) \log_2 |Q_{n,i}|} \\
 &\ge  (n-1)^{1/2} 2^{-n/2} \frac{L_{b_{n,i}} - L_{a_{n,i}}}{5n}\ge  \frac 1 6 n^{-1/2} 2^{-n/2} (L_{b_{n,i}} - L_{a_{n,i}}),
\end{align*}
where we used the elementary inequality $\sqrt{x(n-x)}\ge \sqrt{n-1}$ for all $1\le x\le n-1$.  Therefore, a.s. for all $n\ge n_0(\omega)$ and any covering of $\cZ\cap\cS_n$ by a finite family of intervals  $(Q_{n,i})_{i=1,...,m}$, 
$$
\sum_{i=1}^m  g\left( \frac{|Q_{n,i}|}{2^n}\right) \ge \frac{L(\cup_i Q_{n,i})}{6\sqrt{n 2^{n}} } = \frac{  L_{2^n} - L_{2^{n-1}}}{6\sqrt{n 2^{n}}},
$$
where the last equality follows from the fact that the intervals $Q_{n,i}$ form a covering of the support of $L(dt)$ in $\cS_n$.
The desired lower bound for $\Dimh\cZ_0$ follows if we can show that a.s. 
\begin{align}\label{eq: F_N}
F_N = \sum_{n=1}^N  \frac{L_{2^n} - L_{2^{n-1}}}{\sqrt{n 2^{n}} }
\end{align}
diverges as $N\to +\infty$. {Indeed, this would demonstrate that  the sum \eqref{eq_defdimh} diverges for every $\rho<1/2$ and for any choice of covering.}

 To this end, let us use a second moment argument. As for any fixed $t$, $L_t=|B_t|$ in distribution, one knows that  
\begin{align}
\label{ineg7}
\EE[L_t] = c\sqrt{t} \mbox{ and }  \E[(L_t)^2]=t. 
\end{align}

{An Abel summation manipulation} gives
\begin{align*}
F_N =  \frac{L_{2^N} }{\sqrt{N 2^{N}} } - \frac{L_2 }{\sqrt 2}+ \sum_{n=2}^{N} L_{2^n}{\left( \frac{1}{\sqrt{n2^n}} - \frac{1} {\sqrt{(n+1)2^{n+1}}}\right). }
\end{align*}
Thus, 
\begin{align*}
\E[F_N] \asymp \sqrt{N}.
\end{align*}
By Cauchy-Schwarz inequality, $\E[F_N^2]\ge cN$. To deduce an upper bound for $\E[F_N^2]$,  we use  the expression \eqref{eq: F_N}  {and consider the expectation of the double sum $F_N^2$.} Recall that $T_t$ is the first time  the Brownian motion hits zero after time $t$ and  $L_{t+h}-L_t \le L_{T_t+h}-L_{T_t}= L_h\circ \theta_{T_t}$.   {For $j<\ell$,  applying the strong Markov property at $T_{2^{\ell -1}}$, then at $T_{2^{j-1}}$,  implies  
\begin{align*}
\E[(L_{2^j}-L_{2^{j-1}})(L_{2^{\ell}}-L_{2^{\ell-1}})] \le \E[ (L_{2^j}-L_{2^{j-1}})] \E[L_{2^{\ell -1}}] \le \E[L_{2^{j-1}}]\E[L_{2^{\ell -1}}].
\end{align*} Also, we have
\begin{align*}
\E[(L_{2^n}-L_{2^{n-1}})^2] \le \E[(L_{2^{n-1}})^2].
\end{align*}}
Hence,  we obtain using \eqref{eq: F_N} and \eqref{ineg7}
\begin{align*}
\E[F^2_N] &= \sum_{n=1}^N   {\frac{\E[(L_{2^n}-L_{2^{n-1}})^2]}{n2^n}}+ 2\sum_{\ell=1}^N\sum_{j=1}^{\ell-1}\frac{\E[(L_{2^j}-L_{2^{j-1}})(L_{2^{\ell}}-L_{2^{\ell-1}})]}{\sqrt{j2^j\ell2^\ell}} \\
&\leq  \sum_{n=1}^N  \frac{1}{n} + 2c   \sum_{\ell =1}^{N}\sum_{j=1}^{\ell-1} (j\ell)^{-1/2} \le 3cN.
\end{align*}
Therefore, the Paley-Zygmund inequality yields that there are constants $c, c'>0$ for which for every $N\geq 1$, 
\begin{align*}
\PP(F_N\ge c\sqrt{N}) \ge c'.
\end{align*}
This proves  $\PP(F_\infty=\infty)>0$. By Kolmogorov's zero-one law, $F_\infty=\infty$ almost surely, which completes the proof.

\subsection{Upper mass dimension and logarithmic density of $E_\ga$}

In order to prove that $\mathrm{Den}_{\log} E_\ga =\Dimum E_\gamma=\frac{1}{2}+\ga$, by the relation $\mathrm{Den}_{\log} E \le \Dimum E$ true  for all set $E\subset\RR^+$, it is enough   to prove
$$ \Dimum E_\gamma\leq \frac{1}{2}+\ga \ \ \mbox{ and } \ \ \mathrm{Den}_{\log} E_\ga\ge \frac 1 2 + \ga.$$

\subsubsection{Proof of $\Dimum E_\gamma\leq \frac{1}{2}+\ga$}
 
 We still combine the first moment argument with Lemma \ref{hittingbm}.  For each $\cS_n$,  consider $Q(k,1)\subset \cS_n$ with $k\in\NN$. Observe that $k\in \pix (E_\ga)$ implies that $B$ is within the boundary $E_\ga$ on one of the time intervals $[k-1,k)$ and $[k,k+1)$.  Therefore,
\begin{align*}
\sharp\, \pix(E_\gamma\cap[0,2^n])  \le 2\sum_{p=0}^{n} \sum_{k=2^{p-1}}^{2^p -1}   \indiq_{\cA({p, k,1, \ga})}. 
\end{align*}
Taking expectation,  one obtains by Lemma \ref{hittingbm} that for some finite positive constant $C$, 
\begin{align*}
\E[\sharp\,\pix(E_\gamma\cap[0,2^n])] \le C  \sum_{p=0}^n 2^p 2^{p(\ga -1/2)} = C 2^{n(\ga+1/2)}.
\end{align*}
For $\rho>1/2+\ga$,  the Markov inequality yields that for all $n\ge 1$,
\begin{align*}
\pp\( \sharp\,\pix(E_\gamma\cap[0,2^n]) \ge 2^{n\rho}\)  \le C^{-1}  2^{-n(\rho - \ga-1/2)} 
\end{align*}
which is the general term of a convergent series.  Thus,  an application of Borel-Cantelli Lemma gives almost surely, for all $n$ large enough,  
\begin{align*}
\sharp\,\pix(E_\gamma\cap[0,2^n]) < 2^{n\rho}.  
\end{align*}
Hence, $\Dimum E_\gamma \le \rho$, almost surely.  The desired upper bound follows by letting $\rho$ tend to  $\ga+1/2$.

\subsubsection{Proof of $\mathrm{Den}_{\log} E_\ga\ge \frac 1 2 + \ga$}

Define the sojourn times
 \begin{align*}
 M(t) = \big| \{0<s<t:  |B(s)|\le s^\ga\} \big|,    \quad t> 0.
 \end{align*}
We need to show that
\begin{align}\label{eq:lower}
\limsup_{n\to\infty}  \frac{\log_2 M(2^n)}{n} \ge \ga + \frac 1 2,  \quad \mbox{ a.s. }
\end{align}

We apply a second moment argument.  By Fubini's Theorem, we see that
 \begin{align}\label{eq:1stmoment}
 \EE [M(t)] = \int_0^t \PP(|B(s)|\le s^{\ga}) ds = \int_0^t \PP(|B_1|\le s^{\ga-\frac 1 2})  ds \asymp t^{\ga+\frac 1 2},  \quad t\to + \infty. 
 \end{align}
To conclude, we shall use the following fact yet to be proved: 
\begin{align}\label{eq:2ndmoment}
\limsup_{t\to\infty} \EE\Big[ \left( M(t) t^{-(\ga+\frac 1 2)} \right)^2\Big]<\infty. 
\end{align}
This, combined with \eqref{eq:1stmoment} and the Paley-Zygmund inequality,  leads to
\begin{align*}
\PP(M(t)\ge c\,t^{\ga + \frac 1 2})> c'
\end{align*} 
for all $n$ large enough with some positive constants $c, c'$.  Owing to the triviality of the tail $\sigma$-field of Brownian motion, it follows that 
\begin{align*}
\PP(M(2^n)\ge c 2^{n(\ga + \frac 1 2)} \mbox{ infinitely often}) =1,
\end{align*}
which implies that
\begin{align*}
\limsup_{n\to\infty} \frac{M(2^n)}{2^{n(\ga + \frac 1 2)}}>0, \quad \mbox{ a.s.}
\end{align*}
The desired lower bound \eqref{eq:lower} follows.  

\mk
We end this section with the proof of \eqref{eq:2ndmoment}.  By Fubini's Theorem,
\begin{align}
\EE[M(t)^2] &= 2\int_0^t \int_0^{t-a} \PP\big(|B_a|\le a^{\ga}, |B_{a+b}|\le (a+b)^\ga\big) db\,da \nonumber \\
&\le 2\int_0^t \int_0^{t} \PP\big(|B_a|\le a^{\ga}, |B_{a+b}-B_a|\le 2(a+b)^\ga\big) db\,da \nonumber\\
&= 2\int_0^t \int_0^{t} \PP(|B_a|\le a^{\ga}) \PP\big(|B_b|\le 2(a+b)^\ga\big) db\,da \label{eq:sec3_last}
\end{align}
Note that for any $a>1, b>0$, one has $\PP(|B_a|\le a^\ga)\asymp a^{\ga - \frac 1 2}$ and 
\begin{align*}
\PP\big(|B_b|\le 2(a+b)^\ga\big) \asymp \begin{cases}
1, & \mbox{ if }  b<a^{2\ga}, \\
a^\ga b^{-\frac 1 2}, & \mbox{ if } a^{2\ga}\le b< a, \\
b^{\ga - \frac 1 2}, & \mbox{ if } b\ge a. 
\end{cases}
\end{align*}
Applying this estimate in \eqref{eq:sec3_last}, we arrive at
\begin{align*}
\EE[M(t)^2] \le C (t^{3\ga+\frac 1 2} + t^{2\ga+1} + t^{2\ga+1}) \le C t^{2\ga + 1},
\end{align*}
where we used $\ga<\frac 1 2$ in the second inequality. The proof of \eqref{dimum} is now complete.


\section{Proof of Theorem \ref{main2} :  dimension of all level sets}
\label{sec_main2}

In the previous section, we have proved that almost surely, $\Dimh(\cZ_0) =1/2$. We are going to prove that almost surely,  for every $x\in \R$, $\Dimh(\cZ_x) =1/2$.

For this, we start with the following simple lemma. For any function $f:\R^+\to \R$ and any interval $I\subset \R^+$, the oscillation of $f$ on $I$ is written
\newcommand\osc{\mbox{Osc}}
$$\osc_{I}(f) = \sup_If-\inf_If.$$

\begin{lemma}
\label{lem12}
With probability one, there exists an integer $N\geq 1$ such that for every $n\geq N$, for any integer $\ell\in   [2^{n-1}, 2^{n}-n^{3/2}]$, there exists at least one interval among the consecutive intervals $\{I_k:=[\ell+k,\ell+k+1]\}_{k=0,...,n^{3/2}-1}$ such that   
$\mathrm{Osc}_{I_k}(B)$ is larger than $ \log \log n$.
\end{lemma}
\begin{proof}
Fix $n\geq 1$.
For any integer interval $[k,k+1]\subset \mathcal{S}_n$, let us denote $\tilde p_n=\mathbb{P}(\osc_{[k,k+1]}(B)\geq  \log \log n) >0$. A standard estimate shows that $\tilde p_n =   C_n e^{-( \log \log n)^2/2}/ \log \log n$,  where $C_n$ is a constant uniformly bounded by above and below with respect to $n$.

The Markov property implies that the oscillation of Brownian motion on non-overlapping unit  intervals are independent and identically distributed. Therefore, for any sequence of $\lfloor n^{3/2}/2 \rfloor$ consecutive intervals $I_1$, ..., $I_{\lfloor n^{3/2} /2 \rfloor}$ of length 1, the probability that simultaneously the oscillation of $B$ on each of these intervals   is less than $\log\log n$ is
$p_n:= (1-\tilde p_n)^{\lfloor n^{3/2}/2 \rfloor}.$ One sees that 
$
p_n\asymp e^{ -C_n  n^{3/2 }   e^{-( \log \log n)^2/2}/ \log \log n},
$
 hence goes fast to 0 ($C_n$ is another constant, still bounded away from 0 and $\infty$).

There are $C 2^n/{\lfloor n^{3/2}/2   \rfloor}$  disjoint  sequences of consecutive ${\lfloor n^{3/2} /2 \rfloor}$ integer intervals included in $\mathcal{S}_n$. Hence, by independence, the probability that there is at least one such sequence of  ${\lfloor n^{3/2} /2 \rfloor}$ consecutive intervals such that the oscillation of $B$ on all these   intervals is less than $\log\log n$ is
\begin{equation}
\label{estimpn}
\hat p_n := 1-(1-p_n)^{C2^n/{\lfloor n^{3/2}/2  \rfloor}},
\end{equation}
 which tends exponentially fast to 1. The Borel-Cantelli lemma yields that  there exists some (random) integer $N$ such that for every $n\geq N$, there is always an interval $I$ of length 1 in the ${\lfloor n^{3/2} /2 \rfloor}$ consecutive intervals such that the oscillation of $B$ on $I$ is larger than $\log\log n$. 

As a conclusion, with probability one, every interval of length $n^{3/2}$  (which contains necessarily ${\lfloor n^{3/2} /2 \rfloor}$ consecutive intervals  above) in $\mathcal{S}_n $ contains a subinterval of length 1 on which the oscillation of $B$ is larger than $\log\log n$.
\end{proof}

We also need the following simple fact.

\begin{lemma}\label{lem:6}
Let $E\subset\RR^+$.  One has
\begin{align*}
\Dimh E  = \inf{\Dimh \left(\bigcup_{n\geq 1}\bigcup_{R\in\cR_n} R \right)},
\end{align*}
where the infimum is taken over all families of intervals $(\cR_n)_{n\ge 1}$ such that each  $\cR_n=\{R_{n,i}\}_i$ covers $E\cap\cS_n$ with intervals  of length multiple of $n^2$ with integer boundaries.
\end{lemma}
\begin{proof}
Set $\Dimh E=a$.  For any covering $(\cR_n)_{n\ge 1}$ of $E$,  by monotonicity,  $\Dimh \left(\bigcup_{n\geq 1}\bigcup_{R\in\cR_n} R\right)\ge a$.  To get the converse inequality, observe that the definition of the macroscopic Hausdorff dimension implies that  for any $\ep>0$,  there exists a covering $(\cR^0_n)_{n\ge 1}$ of $E$ with intervals of length multiple of $n^2$ in $\cS_n$ such that $$\sum_{n=1}^\infty \sum_{R\in\cR^0_n} \left(\frac{|R|}{2^n}\right)^{a+\ep}<\infty.$$ Therefore,  $\Dimh\left(\bigcup_{n\geq 1}\bigcup_{R\in\mathcal R_n^0} R\right) \le a+\ep$.  This proves the claim.
\end{proof}

\medskip

We now prove that the macroscopic Hausdorff dimension of Brownian level set is 1/2 over a semi-infinite interval.

\begin{lemma}
\label{lem14}
For every $x_0\in \mathbb{R}$, almost surely,  $\Dimh \mathcal{Z}_x\ge 1/2$ either for all $x\geq x_0$ or for all $x\leq x_0$.
\end{lemma}

The main idea to prove this lemma is to use the oscillation property of the Brownian motion (Lemma \ref{lem12}): 
almost surely, for every  large integer $n$, for all times $t \in \mathcal{S}_n$, the oscillation of $B$ on $[t,t+n^{3/2}]$ is larger than $\log\log n$. 
In particular, when $B_t=0$,  this implies that  one of the two intervals $[-(\log\log n)/2,0]$ or $[0,(\log\log n)/2]$ is included in $B( [t,t+n^{3/2}])$. Hence, either every $x \in [-(\log\log n)/2,0]$ or  every $x \in [0,(\log\log n)/2]$ can be written $B_{t'}$ for some $t'\in [t,t+n^{3/2}]$. This happens, almost surely,  every time $B$ touches 0 on $\mathcal{S}_n$, since the oscillation property of Lemma \ref{lem12} holds for almost every trajectory, for every   $n$.
In other words, each time $B$ touches $0$ ,  the image of $B$ covers  one large interval around 0 (either in $\RR^+$ or in $\RR^-:=(-\infty,0]$), this interval becoming larger  and tending to infinity when $n$ grows to infinity. 

We prove that this implies that the macroscopic dimension of  the level sets $\mathcal{Z}_x$ are necessarily greater than the macroscopic dimension of $\mathcal{Z}_0$.

\begin{proof}[Proof of Lemma \ref{lem14}]
Let $n\ge N$ where $N$ is the integer given in Lemma \ref{lem12}.  We prove the Lemma for $x_0=0$. One knows that  $\Dimh \mathcal{Z}_0=1/2$ almost surely.

 Let  $$\cQ_n = \{Q_{n,i}: i=1, ..., m_n\}$$
    be a finite family of intervals of length multiple of  $n^2$ inside $\cS_n$.   Each $Q_{n,i}$ can be split into a finite number of contiguous intervals $Q$ of length $n^2$.
   
   Consider one of these intervals $Q$ of length $n^2$ ($>n^{3/2}$). If $Q$ intersects $\mathcal{Z}_0$,  by Lemma \ref{lem12}, there exists an interval $I\subset Q$ of length 1 on which the oscillation of $B$ is greater than $\log\log n$. Hence,  (at least) one of the two intervals $[0,\log\log n/2]$ or $[-\log\log n/2,0]$ is necessarily included in  $B({Q})$.  So $Q\cap \mathcal{Z}_x\neq\emptyset$, simultaneously for all $x\in [0,\log\log n/2]$ or for all $x\in [-\log\log n/2,0]$.

We introduce two family of sets by selecting among contiguous intervals $Q$ of length $n^2$ inside $Q_{n,i}$ according to the behavior of $B$ on these intervals: 
\begin{align}
\label{defqni}
\cQ_{n,i}^+ &= \Big\{ Q\subset Q_{n,i}:  |Q|=n^2, \, Q\cap \mathcal{Z}_0 \neq \emptyset, \,  [0,\frac 1 2 \log\log n] \subset B(Q) \Big\},  \nonumber \\
 \cQ_{n,i}^- &= \Big\{ Q\subset Q_{n,i}:  |Q|=n^2, \, Q\cap \mathcal{Z}_0 \neq \emptyset, \,  [-\frac 1 2 \log\log n, 0] \subset B(Q) \Big\}.
\end{align}

Observe that for each $Q\subset Q_{n,i}$, only four cases may occur \begin{itemize}
\item $Q\in \cQ_{n,i}^+ \cap \cQ_{n,i}^-$, if $B(Q)\supset [- \log\log n /2, \log\log n /2]$.
\item
$Q\in \cQ_{n,i}^+ \setminus \cQ_{n,i}^-$ if $B(Q)\supset [0, \log\log n/2]$ and $ [-\log\log n/2, 0)\setminus B(Q)\neq \emptyset$.
\item
$Q\in \cQ_{n,i}^- \setminus \cQ_{n,i}^+$ if $B(Q)\supset [- \log\log n/2,0]$ and $(0, \log\log n/2]\setminus B(Q)\neq \emptyset$.
\item
$Q\cap \cZ_0=\emptyset$. 
\end{itemize}

Finally we set  
\begin{equation}
\label{defqn+}
\begin{split}
\cQ_{n}^+ = \{ Q: Q\in \cQ_{n,i}^+ \mbox{ for some } i \}, \\
 \cQ_{n}^- = \{ Q: Q\in \cQ_{n,i}^- \mbox{ for some } i \}.
  \end{split}
\end{equation}
and
\begin{equation}
\cZ^+ = \bigcup_{n\ge 1}\bigcup_{Q\in\cQ_n^+} Q,  \quad \cZ^- =\bigcup_{n\ge 1} \bigcup_{Q\in\cQ_n^-} Q.
\end{equation}

\begin{remark}
\label{rk3}
The construction ensures that for every $x\in [0,\log\log n/2]$,  if $\cR_n$ is a covering of  $\cZ_x\cap \cS_n$  by intervals $I$ of size that are multiple of $n^2$, then $$\bigcup_{I\in \cR_n} (I+[-n^2, n^2])\supset \bigcup_{Q\in \cQ_{n}^+} Q.$$ 
 The situation for $x\in [-\log\log/2, 0]$ is similar. 
 \end{remark}

Now we choose $\cQ_n=(Q_{n,i})_{i=1,...,m_n}$ so that it forms a covering of $\cZ_0\cap \cS_n$,  for every $n\ge 1$.  
It is obvious from the construction that $$\cZ_0 \subset \bigcup_{n\ge 1}\bigcup_{Q\in\cQ_{n,i}^+\cup \cQ_{n,i}^-} Q=\cZ^+\cup \cZ^-.$$
By the monotonicity of the macroscopic Hausdorff dimension,  either $\Dimh\cZ^+\ge  \frac 1 2$, or $\Dimh\cZ^-\ge \frac 1 2$.  

\medskip

Assume  that $\Dimh\cZ^+ \ge \frac 1 2$.

 Let $x>0$,  by  Remark \ref{rk3},  for any covering $(\cR_n)_{n\ge 1}$ of $\cZ_x$,   $\Dimh\cZ^+ =  \Dimh \cZ^+\cap [e^{e^x}, \infty) \le 
\Dimh (\bigcup_n\bigcup_i R_{n,i}+[-n^2, n^2])$.  It is clear that enlarging each interval $R_{n,i}$  in the covering by two intervals of length $n^2$ on both sides does not change the  macroscopic Hausdorff dimension, as explained below.
\begin{remark}
Observe that by changing $R_{n,i} $ into $\widetilde R_{n,i}:= R_{n,i}+ [-n^2, n^2]$,  one has
$$\sum_i \left( \frac{|  R_{n,i}|}{2^n}\right)^\rho  \leq \sum_i \left( \frac{|  \widetilde R_{n,i}|}{2^n}\right)^\rho \leq 3^\rho \sum_i \left( \frac{|  R_{n,i}|}{2^n}\right)^\rho,$$
  and  the nature (convergence/divergence) of the series involved in the computation of the macroscopic dimension are unchanged. 
\end{remark}
Therefore, by Lemma \ref{lem:6},  for all $x>0$,
\begin{align*}
\frac 1 2 \le \Dimh \cZ^+ = \Dimh \cZ^+\cap [e^{e^x},\infty) \le \inf\Dimh (\bigcup_n\bigcup_i R_{n,i}) = \Dimh \cZ_x,
\end{align*}
where the infimum is taken over all the possible coverings $(\cR_n)_{n\ge 1}$ of $\cZ_x$ with intervals of length multiple of $n^2$ in $\cS_n$.  

The proof for the case $\Dimh \cZ^-\ge \frac 1 2$ is similar.  
\end{proof}

Now we prove   Theorem \ref{main2}.

\begin{proof}

For the upper bound, observe that for every $x\in \R$, $\mathcal{Z}_x$ is ultimately included in $E(\gamma)$ for every $0<\gamma<\frac 1 2$. From this and  Theorem \ref{main} (formula \eqref{dimh}), we deduce that  Dim$_H(\mathcal{Z}_x)\leq 1/2$.

\medskip

We turn to the lower bound.

Let $(x_i)_{i\geq 1}$ be a dense sequence of real numbers. With probability one, the results of Lemma \ref{lem14} apply to all the $x_n$'s simultaneously:  almost surely for every $i$, either $\Dimh(\cZ_x)\ge 1/2\mbox{ for all } x\le x_i$ or $\Dimh(\cZ_x)\ge 1/2\mbox{ for all } x\ge x_i$.

\smallskip

Now, fix $\rho \in [0,1/2)$. Remark that the mapping $(x,\omega)\in \R\times\Omega \longmapsto \nu_\rho^n(\cZ_x)$ is measurable with respect to $\mathcal{B}(\R)\times \mathcal{F}$, where $\mathcal F$ is the minimal augmented $\sigma$-algebra generated by $B$. This simply follows from the fact that $\nu^n_\rho$ is in fact a minimum of a finite number of terms, since the number of possible partitions of $\mathcal{S}_n$ is finite (the intervals $Q_i$ having integer boundaries) and the dependence on $x$ is clearly measurable (indeed, $\mathcal{Z}_x=B^{-1}(\{x\})$).

  For each positive integer $i$, let $X_i = {\bf{1 \!\!\! 1}}_{\mathcal{A}_i}$, where $\mathcal{A}_i $ is the event    
  $$ \mathcal{A}_i =  \left\{\inf \left\{ \sum_{n\ge 1} \nu^n_\rho(\cZ_{x}) : \, {x\le x_i} \right\}  =+\infty\right\}.$$
 By the previous remark,  the sequence of random variables $(X_i)_{i\ge 1}$ is  measurable with respect to the tail $\sigma$-algebra generated by $B$.  Define 
$$Y=\begin{cases}\sup\{x_i: X_i=1\}, &\mbox{ if there exists } i \mbox{ such that } X_i=1, \\ -\infty, &\mbox{ otherwise. } \end{cases}$$
Then $Y$ is also measurable with respect to the tail $\sigma$-algebra. The Kolmogorov 0-1 law implies that  there is a deterministic number $y \in \mathbb{R}\cup\{\pm\infty\}$ such that $Y = y$ a.s.  Three cases are possible: 
 \smallskip
 
{\bf (i)  $y=+\infty$:} almost surely, for all $x\in\RR$, $\sum_{n\ge 1} \nu^n_\rho(\cZ_{x})=\infty$ which implies that $\Dimh \cZ_x\ge \rho$ for all $x$.

\smallskip
 
{\bf (ii) $y=-\infty$:} $X_i=0$ for all $i\in\NN$.   Lemma \ref{lem14} then gives that for all $i$,  $\Dimh \cZ_x\ge 1/2$ for all $x\ge x_i$, hence a contradiction since that would imply $y=+\infty$.

\smallskip
 
{\bf (iii) $y\in\RR$:}  again by Lemma \ref{lem14}, we have :\begin{itemize}
\item
 for all $x_i<y$ and all $x\le x_i$, $\Dimh \cZ_x\ge\rho$ 
 \item  for all $x_i>y$, for all $x\ge x_n$,     $\Dimh \cZ_x\ge 1/2$. 
\item In addition, still by Lemma \ref{lem14} for the fixed deterministic level $y$, almost surely, $\Dimh \cZ_y\ge 1/2$.
\end{itemize}
To conclude, in any case, almost surely, $\Dimh \cZ_x\ge \rho$ for all $x\in\RR$. 

Letting $\rho\uparrow 1/2$ concludes the proof. 
\end{proof}

\section{Concluding remark}

In this article, we considered the large scale structure of certain sojourn sets where a Brownian motion visits ``exceptionally" small values in a large scale dimension sense.  Along the way,  we use quite specific properties of the Brownian motion, for instance, the explicit probability density function of Brownian motion   in Lemma \ref{hittingbm}. It would be of interest to extend the results here to general  L\'evy processes and more general Gaussian processes.  

\section*{Acknowledgments}
The research of S. Seuret was partly supported by the grant ANR MUTADIS ANR-11-JS01-0009. The research of X. Yang was partly supported by DIM grants from R\'egion Ile-de-France. We   thank  Ivan Nourdin,   Giovanni Peccati and Yimin Xiao for stimulating and enjoyable discussions. We are indebted to the anonymous referee for her/his valuable suggestions. 

\bibliographystyle{plain}
\bibliography{xyangbiblio}

\end{document}